%% file: support-lbs.tex
\documentclass[11pt,letterpaper,twoside,reqno,nonamelimits]{amsart}
\usepackage[pagebackref,colorlinks=true,urlcolor=blue,linkcolor=blue,citecolor=blue]{hyperref}
\usepackage[top=1.65in,bottom=1.7in,left=1.5in,right=1.5in]{geometry}
\usepackage{times}
\usepackage{amsmath}
\usepackage{amssymb}
\usepackage{amsfonts}
\usepackage{amscd}
\usepackage{amsthm}
\usepackage{amsxtra}
\usepackage{stmaryrd}
\usepackage[all]{xy}  \CompileMatrices
\usepackage{mathrsfs}
\usepackage{nicefrac}
\usepackage{bbold}
\usepackage{graphicx}
\usepackage{color}
\usepackage{enumerate}
\usepackage{url}
\usepackage{algorithm}

\input{defsN.amsart}
\input{defs.common}

\input{defs.tex}

\begin{document}
\input{abstract.tex}
\input{core.tex}
\input{bib.tex}

\end{document}

%% file: defsN.amsart.tex
\theoremstyle{plain}
\newtheorem{theorem}{Theorem}
\newtheorem*{theorem*}{Theorem}

\newtheorem{proposition}[theorem]{Proposition}
\newtheorem*{proposition*}{Proposition}

\newtheorem{corollary}[theorem]{Corollary}
\newtheorem*{corollary*}{Corollary}

\newtheorem{lemma}[theorem]{Lemma}
\newtheorem*{lemma*}{Lemma}

\newtheorem*{observation*}{Observation}

\newtheorem*{conjecture*}{Conjecture}

\newtheorem*{question*}{Question}

\newtheorem*{questions*}{Questions}

\newtheorem*{problem*}{Problem}

\newtheorem*{problems*}{Problems}

\theoremstyle{definition}
\newtheorem{definition}[theorem]{Definition}
\newtheorem{example}[theorem]{Example}

\newtheorem*{exercise*}{Exercise}

\theoremstyle{remark}

\newtheorem*{remark*}{Remark}

\newtheorem*{remarks*}{Remarks}

\newtheorem*{claim*}{Claim}

\newcommand{\subclass}[1]{}

\newcommand{\enumTi}[1]{\renewcommand{\theenumi}{#1}}

\newcommand{\alphenumi}{\enumTi{\alph{enumi}}}

\newcommand{\romenumi}{\enumTi{\roman{enumi}}}


%% file: defs.common.tex
\newlength{\hspaceforlengthglumpf}

\DeclareMathOperator{\codim}{codim}

\DeclareMathOperator{\supp}{supp}
\DeclareMathOperator{\tr}{tr}

\newcommand{\Id}{\mathrm{Id}}

\DeclareMathOperator{\img}{im}

\DeclareMathOperator{\spn}{span}

\DeclareMathOperator{\rk}{rk}

\newcommand{\orth}{\bot}

\newcommand{\lt}{\left}
\newcommand{\rt}{\right}

\newcommand{\Tp}{{\mspace{-1mu}\scriptscriptstyle\top\mspace{-1mu}}}

\newcommand{\CC}{\mathbb{C}}

\newcommand{\MM}{\mathbb{M}}
\newcommand{\NN}{\mathbb{N}}

\newcommand{\RR}{\mathbb{R}}

\newcommand{\kk}{\mathbb{k}}

\newcommand{\Nm}[1]{{\lt\lVert #1 \rt\rVert}}

\newlength{\algotabbingwidth}


%% file: defs.tex
\numberwithin{theorem}{section}

\DeclareMathOperator{\psdrk}{rk_{\varoplus}}
\DeclareMathOperator{\embrkl}{rk_{\star}}

\newcommand{\poset}{\mathcal P}
\newcommand{\otherposet}{\mathcal Q}
\newcommand{\setofposets}{\mathscr P}
\newcommand{\setofsubspacelattices}{\mathscr L}

\newcommand{\upperCone}{\tilde Q}
\newcommand{\smallCone}{Q_0}
\newcommand{\bigCone}{Q_1}



%% file: abstract.tex

\title[Support-based lower bounds for positive semidefinite rank]{Support-based lower bounds for the positive semidefinite rank of a nonnegative matrix}%
\author{Troy Lee}%
\address{Troy Lee: Centre for Quantum Technologies\\
  National University of Singapore\\
  Singapore 117543}
\email{troyjlee@gmail.com}%
\author{Dirk Oliver Theis}%
\address{Dirk Oliver Theis: Institute of Computer Science\\
  University of Tartu\\
  J.~Liivi 2, 50409 Tartu, Estonia}%
\email{dotheis@ut.ee {\tiny\href{http://dirkolivertheis.blogspot.com}{http://dirkolivertheis.blogspot.com}}}%
\date{Sat Nov 16 01:22:55 EET 2013}%
\begin{abstract}
%
%
%
%
%
%
%
%
%
  The positive semidefinite rank of a nonnegative $(m\times n)$-matrix~$S$ is the minimum number~$q$ such that there exist positive semidefinite $(q\times
  q)$-matrices $A_1,\dots,A_m$, $B_1,\dots,B_n$ such that $S(k,\ell) = \tr A_k^*B_\ell$.
  \\
  The most important lower bound technique on nonnegative rank only uses the zero/non-zero pattern of the matrix.  We characterize the power of lower bounds on
  positive semidefinite rank based on the zero/non-zero pattern.
  \\[1ex]
\end{abstract}
\maketitle


%% file: core.tex

\input{intro}
\input{combopt}

\input{posets}
\input{pf-main}
\input{outlook}

\section*{Acknowledgments}

The second author would like to thank Adam N.~Letchford for discussions about specific positive semidefinite formulations of various optimization problems.  The first author would like to thank Ronald de Wolf for 
helpful discussions and pointing out the connection to nondeterministic quantum communication complexity.



%% file: intro.tex

\section{Introduction}

In this paper, $\kk$ is a subfield of the field~$\CC$ of complex numbers.  For a matrix~$A$ over~$\kk$, we denote its entries by $A(k,\ell)$.  As usual,
$A^*(k,\ell) = \overline{A(\ell,k)}$ is the Hermitian transpose, and $A$ is positive semidefinite if $A$ is Hermitian and all  
eigenvalues are nonnegative.  We let $\kk_+ := \kk\cap \RR_+$ denote the nonnegative numbers in~$\kk$.  A matrix is nonnegative if all its entries are nonnegative.

Let $S$ be an $m\times n$ nonnegative matrix over~$\kk$.  The \textit{nonnegative rank} of~$S$, denoted by $\rk_+(S)$ is the smallest number~$q$ such that there
exists a \textit{nonnegative factorization} of~$S$ of size~$q$, i.e., vectors $\xi_1,\dots,\xi_m, \eta_1,\dots,\eta_n \in \kk_+^q$ such that $S(k,\ell) = \lt( \xi_k
\mid \eta_\ell \rt)$, where the latter is the standard inner product in~$\kk^q$.  Similarly, the \textit{positive semidefinite rank} of~$S$, denoted by $\psdrk(S)$,
is the smallest number~$q$ such that there exists a \textit{positive semidefinite factorization} of~$S$ of size~$q$, i.e., positive semidefinite $(q\times
q)$-matrices $A_1,\dots,A_m, B_1,\dots,B_n$ such that $S(k,\ell) = \tr(A_k^*B_\ell)$, the latter expression being the usual inner product of two square matrices.
These two definitions are examples of the concept of \textit{factorization rank,} where one wishes to write the entries of a matrix~$S$ as inner products of vectors
in some Hilbert space, with diverse restrictions on the set of vectors which are allowed.  

The nonnegative rank is a well-known concept in Matrix Theory, see e.g.~\cite{Markham72,GregoryPullman83,CohenRothblum93}.  Generalizations to other types of
factorizations are of interest there, too, see e.g.~\cite{CohenRothblum93,BeasleyLaffey09}.  In~\cite{BeasleyLaffey09}, the factors $\xi_k$ and $\eta_\ell$ are
required to be in $R^q$, where $R$ is some fixed semiring, e.g., a sub-semiring of $\RR_+$.  To the best of our knowledge, replacing $R^q$ by a cone (in some inner
product space over an ordered field) which is not a product of 1-dimensional cones appears to be a new concept initiated by Gouveia, Parrilo, and
Thomas~\cite{GouveiaParriloThomas12}.

There is a beautiful connection between (1) factorization ranks, (2) linear mappings between convex cones, and (3) combinatorial optimization, which was first noted
by Yannakakis~\cite{Yannakakis91} in 1991 for the nonnegative rank, and later extended by Gouveia, Parrilo, and Thomas~\cite{GouveiaParriloThomas12}.  Driven by
these connections, the last several years have seen a surge of interest in factorization ranks, particularly the nonnegative rank, and recently also the positive
semidefinite rank.  As far as the link to combinatorial optimization is concerned, bounds---upper or lower---on the nonnegative or positive semidefinite rank
provide corresponding bounds on the sizes of linear programming or semidefinite programming formulations of problems.  Finding lower bounds on these factorization
ranks is a difficult task, and draws on methods from combinatorial matrix theory and communication complexity.

For the nonnegative rank, the easiest, most successful, and more or less only method (for an exception 
see~\cite{GillisGlineur11}) for obtaining lower bounds just considers the
support of the matrix.  The \textit{support} of~$S$ is the matrix obtained from~$S$ by replacing every non-zero entry by~$1$.  For an $m \times n$ matrix $S$ whose
support is~$M$, the best lower bound obtainable by considering only the support is
\begin{equation*}
    \min \bigl\{  \rk_+(T)   \bigm|   \supp(T) = M \text{, }\  T \ge 0  \bigr\}.
\end{equation*}
This turns out to be equal to the \textit{Boolean rank} of $M$~\cite{GregoryPullman83}, the smallest $r$ such that there are $r$ dimensional binary vectors $x_1,
\ldots, x_m \in \{0,1\}^r$ and $y_1, \ldots, y_n \in \{0,1\}^r$ satisfying $M(k, \ell) = \vee_{j=1}^r x_k(j) y_\ell(j)$.  The Boolean rank arises in many contexts,
and is also known as \textit{rectangle covering number}~\cite{FioriniKaibelPashkovichTheis13}, \textit{biclique covering number}~\cite{Orlin77} or, after taking
$\log_2$, \textit{nondeterministic communication complexity}~\cite{Yannakakis91}.  Most lower bounds on nonnegative rank actually lower bound the Boolean rank,
including for the recent result showing superpolynomial lower bounds on the size of linear programming formulations of the traveling salesman
problem~\cite{FioriniMassarPokuttaTiwaryDewolf2012}.  Notable exceptions to this rule include results of \cite{Yannakakis91}
and~\cite{KaibelPashkovichTheis10,KaibelPashkovichTheis12}.

%

This paper deals with the question of giving lower bounds for the positive semidefinite rank.  Given the situation for nonnegative rank, it is natural to ask the
following question.

\begin{question*}
  How good can support-based lower bounds for positive semidefinite rank be?
\end{question*}

In the case of the nonnegative rank, there are plenty of examples where the Boolean rank is exponential in the rank.  Moreover, it is not difficult to see that even
the Boolean rank of the support of a rank-3 matrix can be unbounded~\cite{CohenRothblum93}.  In the case of the positive semidefinite rank, we will see that this is
not the case: the best possible support-based lower bound for the positive semidefinite rank coincides with the minimum rank over all matrices with the same
support.


\begin{theorem}\label{thm:main}
  For all 0/1-matrices $M$, we have
  \begin{equation*}
    \min \bigl\{  \psdrk(T)   \bigm|   \supp(T) = M \text{, }\  T \ge 0  \bigr\}
    =
    \min \bigl\{  \rk(T) \bigm| \supp(T) = M \bigr\}
  \end{equation*}
\end{theorem}

The theorem answers completely the question what lower bound information can be gained about the positive semidefinite rank from the zero/non-zero pattern of a
nonnegative matrix: the best possible bound is the minimum possible rank of a matrix with the given zero/non-zero pattern.  De Wolf~\cite{deWolf03} calls this
number the \textit{nondeterministic rank,} and shows that the logarithm of the nondeterministic rank characterizes nondeterministic \textit{quantum} communication
complexity.  We therefore have the pleasing parallel that the logarithm of the best support based lower bound for nonnegative rank is the nondeterministic
communication complexity, while the logarithm of the best support based lower bound on positive semidefinite rank is the nondeterministic quantum communication
complexity.

In the situation of the nonnegative rank, there is a connection between the Boolean rank and embeddings of posets: The Boolean rank of~$M$ is the minimum number of
co-atoms of a truncated Boolean lattice into which a certain poset defined by~$M$ can be embedded.  We prove a corresponding statement for the best-possible
support-based lower bound for the positive semidefinite rank in Section~\ref{sec:posets}.


%% file: combopt.tex

\section{Factorizations}

There is a well-known connection between linear mappings between cones and factorizations of corresponding matrices.  In this section, let $\kk$ be a subfield of
the field~$\RR$ of real numbers.  Let $S$ be a non-negative matrix, and suppose that $S = AX$ for an $(m\times d)$-matrix~$A$ and an a $(d\times n)$-matrix~$X$,
both of rank~$d$.  In other words, we are given a rank-$d$ factorization of~$S$.  Let $\smallCone \subseteq \kk^d$ be the polyhedral cone generated by the columns
of~$X$, and denote by $\bigCone$ the polyhedral cone $\{x\in\kk^d \mid Ax \ge 0\}$.  Clearly, since $S \ge 0$, we have $\smallCone \subseteq \bigCone$.  The rank
condition on $A$ and~$X$ is equivalent to $\smallCone$, $\bigCone$ having dimension~$d$.

A \textit{linear extension} of $\smallCone\subseteq \bigCone$ of size $q$ is a polyhedral cone~$\upperCone$ in some $\kk^s$ with~$q$ facets for which there exists a
linear mapping $\pi \colon \kk^s \to \kk^d$ such that $\smallCone\subseteq \pi(\upperCone) \subseteq \bigCone$.
The following is a well-known fact, going back to Yannakakis.

\begin{theorem}[\cite{Yannakakis91}, c.f.~\cite{FioriniKaibelPashkovichTheis13}]\label{thm:nnegrk-lext}
  The minimum size of a linear extension of $\smallCone\subseteq \bigCone$ equals the nonnegative rank of~$S$.
\end{theorem}

A \textit{positive semidefinite extension} of $\smallCone\subseteq \bigCone$ of size $q$ is the intersection~$\upperCone$ of a linear subspace of some $\MM(q\times
q)$ with the set of all positive semidefinite $(q\times q)$-matrices, for which there exists a linear mapping $\pi \colon \kk^s \to \kk^d$ such that
$\smallCone\subseteq \pi(\upperCone) \subseteq \bigCone$.  The following fact is a straightforward generalization of a recent result by Gouveia, Parrilo, and
Thomas.

\begin{theorem}[\cite{GouveiaParriloThomas12}, c.f.~\cite{TheisBLOG2012:05:20}]\label{thm:psdrk-psdext}
  The minimum size of a positive semidefinite extension of $\smallCone\subseteq \bigCone$ equals the positive semidefinite rank of~$S$.
\end{theorem}

For the reader who wishes to know more about the combinatorial optimization point of view, we recommend~\cite{FioriniKaibelPashkovichTheis13}.

%% file: posets.tex
\section{Poset embedding ranks}\label{sec:posets}

In this section we give a more combinatorial interpretation of the number $\min\{ \psdrk(S) \mid \supp S = \supp M\}$.

\begin{definition}
  Let $S$ be an $(m\times n)$- matrix.  We define the \textit{poset $\poset(S)$ of $S$} as 
  \begin{equation*}
    \poset(S) := \Bigl(\,  \{0\}\times\{1,\dots,m\} \cup \{1\}\times\{1,\dots,n\}\,,\;\; \preceq  \,\Bigr),
  \end{equation*}
  where
  \begin{equation*}
    (i,k) \preceq (j,\ell) \; :\Leftrightarrow\;\; i=0 \land j=1 \land S(k,\ell) \ne 0.
  \end{equation*}
\end{definition}

In other words, $\poset(S)$ is the poset whose Hasse-diagram is the bipartite graph with lower level vertex set the row set of~$S$ and upper level vertex set the
column set of~$S$, and a vertex~$k$ of the lower level adjacent to a vertex~$\ell$ of the upper level if and only if $S(k,\ell) = 0$.

\begin{definition}
  Let $\poset$, $\otherposet$ be posets.  An \textit{embedding} of $\poset$ into~$\otherposet$ is a mapping $j\colon \poset \to \otherposet$ such that $x \le y \iff
  j(x)\le j(y)$ holds for all $x,y\in\poset$.
\end{definition}

\begin{definition}
  Let $S$ be a matrix, $\setofposets$ a set of posets, and $\gimel\colon \setofposets \to \NN$.  We define the \textit{$\setofposets$-embedding rank} of~$S$ as the
  infimum over all $\gimel(\otherposet)$ such that there exists an embedding of $\poset(S)$ into~$\otherposet$.
\end{definition}

As mentioned in the introduction, the Boolean rank of a Boolean matrix~$S$ is equal to the $\setofposets$-embedding rank of $\poset(S)$, with $\setofposets$ the set
of truncated Boolean lattices $\gimel(\otherposet)$ the number of co-atoms of~$\otherposet$~\cite{FioriniKaibelPashkovichTheis13}.

By a \textit{subspace lattice} we mean the lattice of all linear subspaces of $\kk^q$, for some $q\in\NN$.  If $\otherposet$ is the lattice of all subspaces of
$\kk^q$, then we let $\gimel(\otherposet) := q$.  With $\setofsubspacelattices$ the set of all subspace lattices, it is clear that the
$\setofsubspacelattices$-embedding rank, which we denote by $\embrkl(M)$, equals the minimum dimension of a vector space in which there exist subspaces
$U_1,\dots,U_m$ and $V_1,\dots,V_n$ such that
\begin{equation}\label{eq:def-subspace-lattice-embedding}
  U_k\subseteq V_\ell \text{ if, and only if, } S(k,\ell) = 0.
\end{equation}

In the proof of Theorem~\ref{thm:main}, we will prove \textit{en passant} the following proposition.

\begin{proposition}\label{prop:subspacelatticeembedding}\label{obsolete:rem:sle-from-psdfac}
  For all nonnegative $(m\times n)$-matrices~$S$, we have
  \begin{equation*}
    \embrkl(S) = \min \bigl\{ \psdrk(T) \bigm| \supp(T) = \supp(S) \text{, } T \ge 0 \bigr\}.
  \end{equation*}
  More importantly,
  \begin{enumerate}[(a)]
  \item\label{prop:subspacelatticeembedding:a} Every positive semidefinite factorization $S(k,\ell) = \tr(A_k^*B_\ell)$ gives rise to a subspace-lattice embedding
    of~$S$ of the same size by letting $U_k := \ker A_k$ and $V_\ell := \img B_\ell$.
  \item\label{prop:subspacelatticeembedding:b} If $\kk=\RR$, then in~(a) we may assume that $\dim U_k \le (\sqrt{8n+1})/2$ and $\codim V_\ell \le (\sqrt{8m+1})/2$.
  \end{enumerate}
\end{proposition}

It will become clear in the proof that, while the minimum in the subspace-lattice embeddedding rank is always attained by (co-)dimension~1 subspaces, this is not
true for the subspace-lattice embedding arising from a positive semidefinite factorization.

The proposition also shows that the situation for positive semidefinite factorizations mirrors that for nonnegative factorizations.
The subspace-lattice embeddedding rank is the minimum ``size'' $\gimel(\otherposet)$ of a poset~$\otherposet$ of a certain type into which $\poset(S)$ can be
embedded.  The importance of such ``poset embedding ranks'' for factorization ranks has been noted before: it is implicit in~\cite{FioriniKaibelPashkovichTheis13}
that the Boolean rank of a boolean matrix~$S$ is equal to the minimum number of co-atoms in a co-atomic poset\footnote{%
  Recall that a poset is co-atomic if every element is a meet of maximal elements.  The maximal elements are then called co-atoms.
} %
into which $\poset(M)$ can be embedded.



%% file: pf-main.tex

\section{Proof of Theorem~\ref{thm:main} and Proposition~\ref{prop:subspacelatticeembedding}}

In this section we prove Theorem~\ref{thm:main} and Proposition~\ref{prop:subspacelatticeembedding}.  For this, we show the following four lemmas.

\begin{lemma}\label{lem:minrk-leminpsdrk}  
  For all nonnegative matrices $S$ we have
  \begin{equation*}
    \psdrk(S) \ge \min \bigl\{ \rk(T) \bigm| \supp(T) = \supp(M) \bigr\}.
  \end{equation*}
\end{lemma}

\begin{lemma}\label{lem:embrk-le-rk}
  For all matrices $S$
  \begin{equation*}
    \rk(S) \ge \embrkl(S).
  \end{equation*}
  The subspaces $U_k$ in the embedding can be chosen of dimension~$1$, and the subspaces~$V_\ell$ of co-dimension~$1$ (and vice-versa).
\end{lemma}

\begin{lemma}\label{lem:embrk-ge-minpsdrk}
  For all 0/1 matrices $M$, we have
  \begin{equation*}
     \embrkl(M) \ge \min \bigl\{ \psdrk(T) \bigm| \supp(T) = M \text{, } T \ge 0 \bigr\}.
  \end{equation*}
\end{lemma}

\begin{lemma}\label{lem:embrk-le-psdrk}
  Let~$S$ be a nonnegative matrix.  Every positive semidefinite factorization $S(k,\ell) = \tr(A_k^*B_\ell)$ gives rise to a subspace-lattice embedding of~$S$ of
  the same size by letting $U_k := \ker A_k$ and $V_\ell := \img B_\ell$.
\end{lemma}

\begin{lemma}\label{lem:real-psd-factoriz}
  Suppose $\kk=\RR$, and~$S$ is a nonnegative $(m\times n)$-matrix.  If a factorization of~$S$ of size~$q$ exists, then there exists one $S(k,\ell) =
  \tr(A_k^*B_\ell)$ with $\rk A_k \le (\sqrt{8n+1})/2$ and $\rk B_\ell \le (\sqrt{8m+1})/2$ for all $k,\ell$.
\end{lemma}

Theorem~\ref{thm:main} and the equation in Proposition~\ref{prop:subspacelatticeembedding} now follow by sticking together the inequalities.
Proposition~\ref{prop:subspacelatticeembedding}(\ref{prop:subspacelatticeembedding:a}) follows from Lemma~\ref{lem:embrk-le-psdrk}, and
Item\ref{prop:subspacelatticeembedding:b} follows with Lemma~\ref{lem:real-psd-factoriz}.


We start with Lemma~\ref{lem:minrk-leminpsdrk}.  Before we prove it, we note the following easy fact.

\begin{lemma}\label{lem:support-realization}
  Suppose that $S(k,\ell) = \tr A_k^*B_\ell$, $k=1,\dots,m$, $\ell=1,\dots,n$ is a positive semidefinite factorization of~$S$ with matrices of order~$q$.  Then
  there exists a 
  finite union~$H$ of proper sub-varieties of $(\kk^q)^{m+n}$
  such that for any $(\xi_1,\dots,\xi_m,\eta_1,\dots,\eta_n) \in (\kk^q)^{m+n}\setminus H$ we
  have:
  \begin{equation*}
    \bigl( A_k\xi_k \bigm| B_\ell\eta_\ell \bigr) = 0 \quad\iff\quad S(k,\ell) = 0
  \end{equation*}
\end{lemma}

In the case of $\kk\in\{\RR,\CC\}$ one can state more easily that~$H$ is a set of Lebesgue-measure zero.

\begin{proof}[Proof of Lemma~\ref{lem:support-realization}]
  To have $\lt( A_k\xi_k \mid B_\ell\eta_\ell \rt) \ne 0$ for all $(k,\ell)$ with $S(k,\ell)\ne 0$, we need to choose $(\xi,\eta)$ which do not satisfy any of the
  following equations:
  \begin{equation*}
    \lt(\xi_k \mid A_kB_\ell\eta_\ell\rt) = 0\text{; \quad $(k,\ell)$ with $S(k,\ell)\ne 0$}.
  \end{equation*}
  Each of these equations defines a proper sub-variety of $(\kk^q)^{m+n}$, since $0\ne S(k,\ell) = \tr A_k^*B_\ell$ implies $A_k B_\ell \ne 0$.  (This is most
  easily seen by realizing that, for $X := \sqrt A$, $Y := \sqrt B$, we have $\tr A^*B = \Nm{XY}^2$ where $\Nm{Z} := \tr Z^*Z$ refers to the Frobenius- (or
  Hilbert-Schmidt-) norm of the matrix~$Z$.)
\end{proof}

We can now complete the proof of Lemma~\ref{lem:minrk-leminpsdrk}.

\begin{proof}[Proof of Lemma~\ref{lem:minrk-leminpsdrk}.]
  We have to show that for every nonnegative real matrix~$S$ there exists a matrix~$T$ with $\supp(T) = \supp(S)$ and $\psdrk S \ge \rk T$.
  
  Let $S$ be nonnegative and real with $\psdrk S = q$, and let $A_k$, $B_\ell$, $\xi_k$ and $\eta_\ell$, $k=1,\dots,m$, $\ell=1\dots,n$ as in
  Lemma~\ref{lem:support-realization}.  The matrix $T$ defined by $T(k,\ell) := \lt( A_k\xi_k \bigm| B_\ell\eta_\ell \rt)$ has the same support as~$S$ and rank at
  most~$q = \psdrk S$.
\end{proof}


\begin{proof}[Proof of Lemma~\ref{lem:embrk-le-rk}.]
  We have to show $\displaystyle \embrkl(S) \le \rk(S)$ for all matrices $S$.  Let $q := \rk S$.  We give subspaces of a $q$-dimensional vector space~$W$
  satisfying~\eqref{eq:def-subspace-lattice-embedding}.

  For $k=1,\dots,m$, denote by $s_k \in \kk^n$ the vector which constitutes the $k$-th row of $S$, i.e., $s_k = S(k,\dots)^\Tp$, and the let $U_k := \kk s_k$, the
  linear subspace of $\kk^n$ generated by~$s_k$.  The ambient space for our construction is $W := \sum_{k=1}^m U_k$, a vector space of dimension~$q$.  For
  $\ell\in\{1,\dots,n\}$, let $K_\ell$ denote the set of columns indices~$k$ with $S(k,\ell) = 0$, and define
  \begin{equation*}
    V_\ell := \sum_{k \in K_\ell} U_k = \spn \bigl\{ s_k \bigm| S(k,\ell) = 0 \bigr\}.
  \end{equation*}
  Clearly, $U_1,\dots,U_m$, $V_1,\dots,V_n$ are linear subspaces of a real vector space of dimension~$q$.  Moreover, by construction, we have $U_k \subseteq V_\ell$
  whenever $S(k,\ell) = 0$.  But since
  \begin{equation*}
    V_\ell \subseteq \{ x\in \kk^n \mid x(k) = 0\;\; \forall\; k\in K_\ell \},
  \end{equation*}
  we have that $S(k,\ell) \ne 0$ implies $U_k \nsubseteq V_\ell$, and we conclude~\eqref{eq:def-subspace-lattice-embedding}.
\end{proof}


\begin{proof}[Proof of Lemma~\ref{lem:embrk-ge-minpsdrk}.]
  We have to show $\displaystyle \embrkl(M) \ge \min \bigl\{ \psdrk(T) \bigm| \supp(T) = M \text{, } T \ge 0 \bigr\}$ for all 0/1 matrices $M$.  For this, from
  subspaces of~$\kk^q$ satisfying~\eqref{eq:def-subspace-lattice-embedding} with $S$ replaced by~$M$, we constuct a matrix~$T$ and a positive semidefinite
  factorization with matrices of order~$q$.

  Let $U_1,\dots,U_m$, $V_1,\dots,V_n$ such a collection of subspaces.  Fix any inner product of $\kk^q$, and denote by $A_k$ the matrix of the orthogonal
  projection of $\kk^q$ onto $U_k$ and by $B^\orth_\ell$ the matrix of the orthogonal projection of $\kk^q$ onto $V_\ell$, by $\Id$ the $q\times q$ identity matrix,
  and let $B_\ell := \Id - B^\orth_\ell$.  Clearly $A_k$ and $B_\ell$ are positive semidefinite, and we have $A_kB_\ell = 0$ if and only if $M(k,\ell) = 0$.  Thus,
  $T$ defined by $T(k,\ell) := \tr A_k^*B_\ell$ is a matrix with $\supp(T) = M$, and $A_k$, $B_\ell$ a positive semidefinite factorization.
\end{proof}


\begin{proof}[Proof of Lemma~\ref{lem:embrk-le-psdrk}.]
  From a positive semidefinite factorization with matrices of order~$q$, we will construct subspaces of~$\kk^q$
  satisfying~\eqref{eq:def-subspace-lattice-embedding}.

  Let a positive semidefinite factorization of~$S$ be given, i.e., let $A_1,\dots,A_m$, $B_1,\dots,B_n$ be $q\times q$ real positive semidefinite matrices with
  $S(k,\ell) = \tr A_k^*B_\ell$.  Now, for positive semidefinite matrices $A$, $B$, the two statements $\tr A^*B =0$ and $AB=0$ are equivalent.  But $A_kB_\ell = 0$
  is equivalent to $U_k := \img A_k \subseteq \ker B_\ell =: V_\ell$.
\end{proof}


\subsection{The case $\kk=\RR$}

For positive semidefinite matrices with real entries, the following is well-known.

\begin{lemma}[E.g.~\cite{barvinok02book}]\label{lem:eqns-in-real-psd}
  Let $A_1,\dots,A_m$ be square matrices, and $\alpha_1,\dots,\alpha_m$ numbers.  If there exists a real positive semidefinite matrix~$X$ such that $\tr(A_j^*X) =
  \alpha_j$ for $j=1,\dots,m$, then there exists such a matrix~$X$ with rank at most $(\sqrt{8m}+1)/2$.
\end{lemma}

\begin{proof}[Proof of Lemma~\ref{lem:real-psd-factoriz}.]
  This lemma is an easy consequence of Lemma~\ref{lem:eqns-in-real-psd}.  We leave the easy details to the reader.
\end{proof}

\subsection{A corollary}
We close this section by stating the following combinatorial corollary of Theorem~\ref{thm:main}

\begin{corollary}\label{cor:triangular-rk-lb}
  Let $S$ be a nonnegative matrix.  The triangular rank of~$S$ is a lower bound to the positive semidefinite rank of $S$.
  \qed
\end{corollary}


%% file: outlook.tex

\section{Outlook}
As we have shown, support-based lower bounds on the positive semidefinite rank of a matrix will always be at most the rank.  (In fact, one might wonder whether the
rank of a matrix is always an upper bound on its positive semidefinite rank, but for each~$r\ge 3$, Corollary~4.16 in~\cite{GouveiaParriloThomas12} gives families
of matrices with rank~$r$ and unbounded positive semidefinite rank.)
We illustrate how lower bounds which move beyond considering the support might be based on subspace-lattice embeddings via
Proposition~\ref{prop:subspacelatticeembedding}.


\begin{example}
  With $\kk:=\RR$, consider the $(n\times n)$-matrix $S_n$ where $S_n(i,j)= (i-j-1)(i-j-2)/2$.
  We have $\rk S_n = 3$ for all $n$, which follows from the expansion
 \begin{equation*}
    (i-j-1)(i-j-2) = (i^2-3i+1) + (j^2+3j+1) - (2ij),
\end{equation*}
as each term in parenthesis can be expressed as a rank one matrix.  

  We conjecture that the positive semidefinite rank of $S_n$ grows unboundedly with~$n$.  (Note that the bound in~\cite[Corollary~4.16]{GouveiaParriloThomas12} does
  not apply since~$S_n$ is not the slack-matrix of a polytope.) We can prove the following.
  \begin{claim*}
    If $n \ge 6$, the positive semidefinite rank of~$S_n$ is at least~4.
  \end{claim*}
  
  \begin{proof}[Proof of the claim.]
    By considering the upper-left $6\times 6$ submatrix, it suffices to prove the claim for $n=6$:
    \begin{equation*}
      S_6 =
      \begin{bmatrix}
        1 & 3 & 6 &10 &15 &21 \\
        0 & 1 & 3 & 6 &10 &15 \\
        0 & 0 & 1 & 3 & 6 &10 \\
        1 & 0 & 0 & 1 & 3 & 6 \\
        3 & 1 & 0 & 0 & 1 & 3 \\
        6 & 3 & 1 & 0 & 0 & 1 
      \end{bmatrix}
    \end{equation*}
    By contradiction, assume that $A_1,\dots,A_6$, $B_1,\dots,B_6$ is a positive semidefinite factorization of~$S_6$ of order~$3$.
    
    Let $U_k$, $V_\ell$ be subspaces of $\RR^3$ as in Proposition~\ref{prop:subspacelatticeembedding}.  Since for $k\ge 3$, the $k$th row contains zeros and
    non-zeros, we have $\dim U_k \ge 1$ for these~$k$.  For the same reason, we have $\dim V_\ell \le 2$ for $\ell \le 4$.  If we had $\dim U_k = 2$ for any $k\ge
    3$, then, for $\ell,\ell'$ with $S_6(k,\ell)=S_6(k,\ell')=0$, it would follow that $V_\ell = V_{\ell'}$, which is impossible since the $\ell$th column differs
    from the $\ell'$th.  Thus we conclude that $\dim U_k = 1$ for $k\ge 3$.  Similarly, we have $\dim V_\ell = 2$ for $\ell \le 4$.
    
    But this means that $A_k$, $k\ge 3$, and $B_\ell$, $\ell \le 4$, are rank-1 matrices.  Choose vectors $u_k,v_\ell \in \RR^3$, $k=3,\dots,6$, $\ell=1,\dots,4$,
    such that $A_k = u_ku_k^\Tp$, and $B_\ell = v_\ell v_\ell^\Tp$.  For these $k,\ell$, we have
    \begin{equation*}
      S_6(k,\ell) = \tr(u_ku_k^\Tp v_\ell v_\ell^\Tp) = (u_k^\Tp v_\ell)^2 = Y(k,\ell)^2,
    \end{equation*}
    where we define the rank-3 matrix $Y(k,\ell) := u_k^\Tp v_\ell$.  Since $Y(k,\ell) = \pm \sqrt{S_6(k,\ell)}$, we may enumerate all the $2^9$ possible choices
    for~$Y$.  Doing this, we see that all possible choices for~$Y$ have rank at least~4, so no such~$Y$ can exist, a contradiction.
    (We note that, independently, the technique based on entry-wise square roots has been used and further developed in~\cite{GouveiaRobinsonThomas13}.)
  \end{proof}
\end{example}

This example shows how using additional structure of a positive semidefinite factorization---for example that if~$S$ has a rank-one semidefinite factorization of
dimension~$k$ then there is a matrix~$Y$ of rank~$k$ whose entrywise square is~$S$---can lead to improved lower bounds.  The following concrete problems motivate
finding more general methods that can show positive semidefinite rank lower bounds larger than the rank.



For a real matrix~$S$, can the positive semidefinite rank over $\kk:=\RR$ be larger than the positive semidefinite rank over $\kk := \CC$?  This mirrors the
corresponding problem posed by Cohen \& Rothblum~\cite[Section~5]{CohenRothblum93} (cf.~\cite{BeasleyLaffey09}) regarding the nonnegative rank over the reals of
rational matrices.



%% file: bib.tex
\providecommand{\bysame}{\leavevmode\hbox to3em{\hrulefill}\thinspace}
\providecommand{\MR}{\relax\ifhmode\unskip\space\fi MR }
\providecommand{\MRhref}[2]{%
  \href{http://www.ams.org/mathscinet-getitem?mr=#1}{#2}
}
\providecommand{\href}[2]{#2}